\documentclass[leqno,centertags]{amsart} 
\usepackage{enumerate} 
\usepackage{mathrsfs}
\usepackage{regrng}

\begin{document}

\title{Biordered sets of regular rings} 
\author{James Alexander
       \and
       E.~Krishnan} 
\date{}

\begin{abstract}
  The set of idempotents of a regular semigroup is given an abstract
  characterization as a \emph{regular biordered set} in \cite{kss},
  and in \cite{pa} it is shown how a biordered set can be associated
  with a complemented modular lattice. Von Neumann has shown
  earlier that any complemented modular lattice of order greater than
  3 can be realized as the lattice of principal right ideals of a
  regular ring (see \cite{vn}). Here we try to connect these ideas to
  get a characterization of the biordered sets of a class of regular
  rings.
\end{abstract}

\maketitle

\section{Introduction}

By a \textit{semigroup}, we mean a non-empty set $S$ with an
associative binary operation, $(x,y)\mapsto xy$, from $S\times S\to
S$. An element $x$ of a semigroup $S$ is said to be \textit{regular}
if there exists an element $x'$ in $S$ with $xx'x=x$.  If $x'$
satisfies the equation $x'xx'=x'$ also, then $x'$ is called a
\textit{generalized inverse} of $x$. It is not difficult to show that
every regular element has a generalized inverse, for if $x'$ satisfies
the equation of regularity, then $x''=x'xx'$ satisfies both equations
for generalized inverses. A semigroup in which all the elements are
regular is called a regular semigroup. A ring $R$ is said to be a
\emph{regular ring} if its multiplicative semigroup is regular.

Any regular semigroup has a rich supply of \textit{idempotents}, that
is, elements $e$ for which $e^2=e$. In \cite{kss}, the set of
idempotents of a regular semigroups is given an abstract
characterization as a partial algebra with two quasi-orders, which is
termed a \textit{regular biordered set}.  We give here a few essential
notions of biordered sets. Details can be found in \cite{kss}. Let $E$
be a non-empty set in which a partial binary operation is defined.
(This means the product $ef$ is defined only for certain pairs $e, f$
of elements of $E$.) We define two relations $\omr$ and $\oml$ on $E$
by
\begin{equation*}
  \oml=\{(e,f)\in E\times E\colon ef=e\} 
  \quad\text{and}\quad
  \omr=\{(e,f)\in E\times E\colon fe=e\}
\end{equation*}
One of the axioms of a biordered set is that these relations are
\textit{quasi-orders}, that is, they are reflexive and transitive. Note
that this means the relation $\om$ defined by 
\begin{equation*}
  \om=\oml\medcap\omr
\end{equation*}
is a partial order on $E$. For $e$ in $E$, we define
\begin{equation*}
  \oml(e)=\{f\in E\colon f\rel\oml e\} 
  \quad\text{and}\quad
  \omr(e)=\{f\in E\colon f\rel\omr e\}
\end{equation*} 
And for $e$, $f$ in $E$, we define
\begin{equation*}
  \mset ef=\oml(e)\medcap\omr(e)
\end{equation*}
Also, for $e$, $f$ in $E$, we define the \textit{sandwich set} of $e$
and $f$ by
\begin{equation*}  
  \swset ef=
  \{h\in\mset ef\colon g\preceq h\;\;\text{for all}\;\;g\in\mset ef\}
\end{equation*}
where $\preceq$ is the quasi-order defined by
\begin{equation*}  
  g\preceq h\iff eg\rel\omr eh, gf\rel\oml hf
\end{equation*}
The regularity condition on a biordered set is that 
\begin{equation*}
  \swset ef\ne\emptyset
  \;\;\text{for all $e$ and $f$ in $E$}
\end{equation*}

We first see how certain properties of the idempotents in a regular
ring can be formulated in biorder terms and then later show that these
properties actually characterize the biordered set of a ring of
matrices.

\section{Idempotents in a regular ring}

Let $R$ be a ring with unity and let $E$ be the set of
idempotents of $R$. It is easily seen that if $e$ is an idempotent in
$R$, then $1-e$ is also an idempotent in $R$. Thus if we denote $1-e$ by
$\annid e$, then we have a map $e\mapsto\annid e$ with $\dannid
e=e$. We first prove some elementary properties of this map in terms of
the biorder relations of $E$.
\begin{prop}\label{annid}
  Let $E$ be the set of idempotents of a ring with unity and
  for each $e$ in $E$, let $\annid e=1-e$. Then for $e$, $f$ in $E$,
  we have the following:
  \begin{mathenum}
  \item $f\rel\oml e$ if and only if $\annid e\rel\omr\annid f$
  \item $f\rel\oml\annid e$ if and only if $fe=0$
  \end{mathenum}
\end{prop}
\begin{proof}
  If $f\rel\oml e$, then by definition of $\oml$, we have $fe=f$, so
  that 
  \begin{equation*}
    \annid f\annid e=(1-f)(1-e)=1-e-f+fe=1-e=\annid e
  \end{equation*}
  and so $\annid e\rel\omr\annid f$. Conversely, if  
  $\annid e\rel\omr\annid f$, then $\annid f\annid e=\annid e$, so that
  \begin{equation*}
    fe=(1-\annid f)(1-\annid e)
    =1-\annid e-\annid f+\annid f\annid e
    =1-\annid f=f
  \end{equation*}
  and so $f\rel\oml e$. This proves (i).

  To prove (ii), first let  $f\rel\oml\annid e$. Then $f\annid e=f$,
  so that
  \begin{equation*}
    fe=f(1-\annid e)=f-f\annid e=0
  \end{equation*}
  Conversely, if $fe=0$, then 
  \begin{equation*}
    f\annid e=f(1-e)=f-fe=f
  \end{equation*}
  so that $f\rel\oml\annid e$
\end{proof}

The condition $fe=0$ can also be formulated in biorder terms in any
regular semigroup with a zero element, using the idea of the $\mathsf{M}$-set
defined earlier.
\begin{lem}\label{prod0}
  Let $S$ be a regular semigroup with zero and let $e$ and $f$ be idempotents
  in $S$. Then $ef=0$ if and only if\/ $\mset ef=\{0\}$.
\end{lem}
\begin{proof}
  First suppose that $ef=0$ and let $g\in\mset ef$.  Then
  by definition, $g$ is an idempotent in $S$ with $ge=g=fg$ so that
  \begin{equation*}
    g=g^2=(ge)(fg)=g(ef)g=0
  \end{equation*}
  since $ef=0$. Conversely, suppose $\mset ef=\{0\}$. Since
  $S$ is regular, the element $ef$ in $S$ has a generalized inverse
  $x$ in $S$.  Let $g=fxe$. Then
  \begin{equation*}
    g^2=f(xefx)e=fxe=g
  \end{equation*}
  so that $g$ is an idempotent.  Also, $ge=g=fg$ so that
  $g\in\mset ef$. Hence $g=0$ and so
  \begin{equation*}
    ef=(ef)x(ef)=e(fxe)f=egf=0
  \end{equation*}
  This completes the proof.
\end{proof}

Using this, the second part of Proposition~\ref{annid} can be reformulated as
follows:
\begin{cor}
  For idempotents $e$, $f$ in a regular ring with unity,
  $f\rel\oml\annid e$ if and only if $\mset ef=\{0\}$
\end{cor}

We next note that if $e$ and $f$ are idempotents in a ring with
$ef=fe=0$, then $e+f$ is also an idempotent. This property can be
characterized in biorder terms. We first note that the
conditions $ef=fe=0$ are equivalent to the conditions 
$e\rel\oml\annid f$ and $f\rel\oml\annid e$, by
Proposition~\ref{annid}(ii) and the relation $e\rel\oml\annid f$ is
equivalent to $f\rel\omr\annid e$, by Proposition~\ref{annid}(i). It
follows that $ef=fe=0$ iff $f\rel\om\annid e$. Also, if this condition
holds, then the idempotent $e+f$ can be characterized in terms of
sandwich sets. For this, we make use of the fact that if $E$ is the
biordered set of idempotents of a regular semigroup $S$, then
\begin{equation*}
  \swset ef=\{h\in E\colon fhe=h\;\;\text{and}\;\;ehf=ef\;\;\text{in}\;\;S\}
\end{equation*}
(cf.Theorem 1.1 of \cite{kss}). 
\begin{prop}
  Let $E$ be the biordered set of idempotents of a regular ring with
  unity and for each $e$ in $E$, let $\annid e=1-e$. For $e$, $f$ in
  $E$ if $f\rel\om\annid e$, then there is a unique idempotent in $E$
  which belongs to both the sandwich sets
  $\swset{\annid e}{\annid f}$ and $\swset{\annid f}{\annid e}$
\end{prop}
\begin{proof}
  Let $f\rel\om\annid e$. Then as noted above, we have
  $ef=fe=0$. Hence 
  \begin{equation*}
    (e+f)^2=e+f+ef+fe=e+f
  \end{equation*}
  so that $e+f$ is in $E$. Let $h=1-(e+f)$ so that $h$ is also in $E$.
  Now
  \begin{equation*}
    \annid e\annid f=(1-e)(1-f)=1-e-f=h
  \end{equation*}
  since $ef=0$ and similarly
  \begin{equation*}
    \annid f\annid e=h
  \end{equation*}
  since $fe=0$. Hence
  \begin{equation*}
    \annid fh\annid e=\annid f(\annid f\annid e)\annid e
    =\annid f\annid e=h
  \end{equation*}
  since $\annid e$ and $\annid f$ are idempotents. Again,
  \begin{equation*}
    \annid eh\annid f=\annid e(\annid e\annid f)\annid f=\annid e\annid f
  \end{equation*}
  Since $E$ is the biordered set of the idempotents of the
  multiplicative semigroup of $R$, which is regular, it follows from
  the comments preceding the result that $h$ is in 
  $\swset{\annid e}{\annid f}$. Similar computations show that $h$ is also in
  $\swset{\annid f}{\annid e}$. 

  To prove uniqueness, let $g$ be an element in $E$ belonging to both
  these sandwich sets. Then
  \begin{equation*}
    \annid eg\annid f=\annid e\annid f
  \end{equation*}
  since $g$ is in $\swset{\annid e}{\annid f}$ and
  \begin{equation*}
    \annid eg\annid f=g
  \end{equation*}
  since $g$ is in $\swset{\annid f}{\annid e}$. Hence 
  $g=\annid e\annid f=h$.
\end{proof}

Another property of the biordered set of a regular ring is linked to
the ideal theory of regular rings. It is well known that in a regular
ring, every principal left or right ideal is generated by an
idempotent, and that the set of principal left ideals and the set of
principal right deals of a regular ring with unity form dually
isomorphic complemented modular lattices (see \cite{vn}, \cite{bj}).
We next show that under certain conditions discussed above, a
biordered set can be realized as the biordered set of a regular
semigroup whose principal left ideals and principal right ideals form
dually isomorphic complemented modular lattices.

\section{Strongly regular Baer semigroups}

We start by observing that in any semigroup $S$ with zero, we can
define \textit{left  annihilator} of an element $s$ by
\begin{equation*}
  \lann s=\{x\in S\colon xs=0\}
\end{equation*}
and the \textit{right annihilator} of $s$ by
\begin{equation*}
  \rann s=\{x\in S\colon sx=0\}
\end{equation*}
The semigroup $S$ is said to be a \textit{strongly regular Baer
  semigroup} if the set of left annihilators of elements of $S$ is
equal to the set of principal left ideals of $S$ and the set of right
annihilators is equal to the set of principal right ideals of $S$. It
can be shown that the multiplicative semigroup of a regular ring
with unity is a strongly regular Baer semigroup (\cite{bj}). Also, a
strongly regular Baer semigroup is a regular semigroup, in the sense
defined earlier.  

We now show that if a biordered set satisfies some of the properties
discussed in the previous section, then it can be realized as the
biordered set of a strongly regular Baer semigroup.
\begin{thm}\label{srbsg}
  Let $E$ be a regular biordered set with the following
  properties. 
  \renewcommand{\theenumi}{{\normalfont(\textsf{E\arabic{enumi}})}}
  \renewcommand{\labelenumi}{\theenumi}
  \renewcommand{\theenumii}{{\normalfont(\roman{enumii})}}
  \renewcommand{\labelenumii}{\theenumii}
  \begin{enumerate}
  \item There exists an element $0$ in $E$ such
    that $0\rel\om e$ for each $e$ in $E$\label{e1}
  \item There exists a map 
    $e\mapsto\annid e$ satisfying the following conditions:\label{e2}
    \begin{enumerate}
    \item $\dannid e=e$ for each $e$ in $E$\label{e21}
    \item $f\rel\oml e$ if and only if $\annid e\rel\omr\annid f$ for
      $e$, $f$ in $E$\label{e22}
    \item $f\rel\oml\annid e$ if and only if $\mset fe=\{0\}$ for $e$,
      $f$ in $E$\label{e23}
    \end{enumerate}
 \end{enumerate}
  Then there exists a strongly regular Baer semigroup $S$ such that
  $E$ is the biordered set of idempotents of $S$. 
\end{thm}
To prove this result, we make use of a couple of lemmas. First we show that
in any biordered set satisfying the above conditions, the duals of 
these conditions also hold.
\begin{lem}\label{dual}
  Let $E$ be a biordered set satisfying \ref{e1} and
  \ref{e2}. Then $E$ satisfies the following conditions also:
  \begin{mathenum}
  \item there exists $1$ in $E$ such that $e\rel\om1$ for each
    $e$ in $E$.
  \item $f\rel\omr e$ if and only if $\annid e\rel\oml\annid f$, for
    $e$, $f$ in $E$
  \item $f\rel\omr\annid e$ if and only if $\mset ef=\{0\}$, for
    $e$, $f$ in $E$
\end{mathenum}
\end{lem}
\begin{proof}
  We first prove (ii). Let $e$ and $f$ be elements of $E$ with
  $f\rel\omr e$. By \ref{e21}, we have $e=\dannid e$ and $f=\dannid
  f$, so that we have $\dannid f\rel\omr\dannid e$. By \ref{e22}, this
  gives $\annid e\rel\oml\annid f$. On the other hand if we have
  $\annid e\rel\oml\annid f$, then from \ref{e22} we get $\dannid
  f\rel\omr\dannid e$ which gives $f\rel\omr e$, by \ref{e21}.
  
  Now to prove (i), let $1=\annid 0$. Then for each $e$ in $E$,
  since $0\rel\oml\annid e$ by \ref{e1}, we have 
  $\dannid e\rel\omr\annid 0=1$, by \ref{e22}, so that $e\rel\omr 1$,
  using \ref{e21}. Again, since $0\rel\omr\annid e$ by \ref{e1}, we
  have $\dannid e\rel\oml\annid0$, by what we have proved above and so 
  $e\rel\oml1$. Thus $e\rel\om1$.

  To prove (iii), first let $e$ and $f$ be elements of $E$ with
  $f\rel\omr\annid e$. Then by what we have proved above, we get
  $\dannid e\rel\oml\annid f$ and hence $e\rel\oml\annid f$, using
  \ref{e21}. By \ref{e23}, this gives $\mset ef=\{0\}$. Conversely
  suppose $e$ and $f$ are elements of $E$ with $\mset ef=\{0\}$.
  Then from \ref{e23}, we get $e\rel\oml\annid f$ and hence $\dannid
  f\rel\omr\annid e$, from \ref{e22}; that is, $f\rel\omr\annid e$,
  using \ref{e21}.
\end{proof}

Now if $E$ is a regular biordered set, then there exists a regular
semigroup $S$ with $E$ as its set of idempotents and which is
idempotent generated, in the sense that every element of $S$ is a
product of elements of $E$ (see Section 6 of \cite{kss}). It is easy
to see that if $E$ is a regular biordered set satisfying (E1) and
(E2), then 0 is the zero and 1 is the identity of every idempotent
generated regular semigroup with $E$ as its set of idempotents. We
next show that for each $e$ in $E$, the generator of the annihilators
of $e$ in such a semigroup is $\annid e$.  
\begin{lem}\label{ann}
  Let $E$ be a regular biordered set satisfying \ref{e1} and
  \ref{e2} and $S$ be a regular idempotent generated semigroup with
  $E$ as its biordered set of idempotents. Then for each $e$ in
  $E$, we have $\lann e=S\annid e$ and $\rann e=\annid eS$.
\end{lem}
\begin{proof}
  First note that since $S$ is idempotent generated, the element
  $0$ of $E$ is the zero of $S$. Let $e$ be an element of
  $E$ and let $x$ be an element of $S$ which belongs to
  $\lann e$ so that $xe=0$. Since $S$ is regular, there exists $x'$
  in $S$ with $xx'x=x$. Let $f=x'x$ so that $f$ is an element of
  $E$ with
  \begin{equation*}
    xf=xx'x=x
  \end{equation*}
  Now
  \begin{equation*}
    fe=(x'x)e=x'(xe)=0
  \end{equation*}
  so that $\mset fe=\{0\}$, by Lemma~\ref{prod0}. Hence $f\rel\oml\annid e$,
  by \ref{e23}. So, $f\annid e=f$, by definition of $\oml$. This gives
  \begin{equation*}
    x\annid e=(xf)\annid e=x(f\annid e)=xf=x
  \end{equation*}
  Thus $x=x\annid e\in S\annid e$. It follows that 
  $\lann e\subseteq S\annid e$. 

  To prove the reverse inclusion, let $x\in S\annid e$ so that
  $x\annid e=x$. Let $f$ be defined as before. Then
  \begin{equation*}
    f\annid e=(x'x)\annid e=x'(x\annid e)=x'x=f
  \end{equation*}
  so that $f\rel\oml\annid e$, by definition and so $\mset fe=\{0\}$,
  by \ref{e23}. Hence $fe=0$, by Lemma~\ref{prod0}, so that
  \begin{equation*}
    xe=(xf)e=x(fe)=0
  \end{equation*}
  Thus $x\in\lann e$ and it follows that $S\annid e\subseteq\lann e$.
  So $S\annid e=\lann e$. A dual argument using Lemma~\ref{dual}
  proves the result for right annihilators.
\end{proof}
Now we can prove our theorem.
{\renewcommand{\proofname}{\textsc{Proof of the Theorem}}
\begin{proof}
  Since $E$ is a regular biordered set, there exists a regular
  idempotent generated semigroup $S$ with $E$ as the biordered set  of
  idempotents, as noted earlier. We will show that $S$ is a
  strongly regular Baer semigroup. 

  To show that the left annihilator of each element is a principal
  left ideal in $S$, let $x$ be an element of $S$ and consider
  $\lann x$. Since $S$ is regular, there exists $x'$ in $S$ with
  $xx'x=x$. Let $e=xx'$ so that $e$ is an element of $S$ with
  $ex=x$. We can show that $\lann x=\lann e$. For if $y\in\lann x$, then
  $yx=0$ so that
  \begin{equation*}
    ye=y(xx')=(yx)x'=0
  \end{equation*}
  and so $y\in\lann e$; on the other hand, if $y\in\lann e$, so that
  $ye=0$, then 
  \begin{equation*}
    yx=y(ex)=(ye)x=0
  \end{equation*}
  and so $y\in\lann x$. Thus $\lann x=\lann e$ and by Lemma~\ref{ann}, we
  have $\lann e=S\annid e$. So, $\lann x=S\annid e$. 

  On the other hand, we can show that every principal left ideal in
  $S$ is the left annihilator of an element in $S$. Let $x$ be an
  element of $S$ and let $x'$ be an element of $S$ with
  $xx'x=x$. Then $e=x'x$ is an idempotent with
  \begin{equation*}
    Se=Sx'x\subseteq Sx
    \quad\text{and}\quad
    Sx=Sxx'x\subseteq Sx'x=Se
  \end{equation*}
  so that $Sx=Se$. Now by
  \ref{e21}, we have $\dannid e=e$, so that $Se=S\dannid e$. 
  Also, by Lemma~\ref{ann}, we have $S\dannid e=\lann{\annid e}$. Thus
  \begin{equation*}
    Sx=Se=S\dannid e=\lann{\annid e}
  \end{equation*}
  A dual argument proves the corresponding results for principal right
  ideals and right annihilators. Hence, by definition, $S$ is a
  strongly regular Baer semigroup.
\end{proof}}

Now the set of principal left ideals and the set of principal right
ideals of a strongly regular Baer semigroup can be shown to be
complemented modular lattices which are dually isomorphic (see
\cite{bj}). Also, the partially ordered set of principal left ideals 
and the partially ordered set of principal right ideals of any regular
semigroup are isomorphic to the quotients of its biordered set  
by certain equivalence relations, as indicated below.

Let $S$ be a regular semigroup and let $E$ be the biordered set of its
idempotents. We define the relations $\gle$ and $\gre$ on $E$ by
\begin{equation*} 
  \gle=\oml\medcap\,(\oml)^{-1}  
  \quad\text{and}\quad
  \gre=\omr\medcap\,(\omr) ^{-1}
\end{equation*}
It is easily seen that the relations $\gle$ and $\gre$ are
equivalences on $E$ and hence partition $E$. For each $e$ in $E$, we
denote the $\gle$-class containing $e$ by $\gle(e)$ and the
$\gre$-class containing $e$ by $\gre(e)$. The set of all
$\gle$-classes is denoted by $E/\gle$ and the set of all
$\gre$-classes by $E/\gre$.  Now for $e$ and $f$ in $E$, if
$e\rel\oml f$ and $e'\in\gle(e)$ and $f'\in\gle(f)$, then 
$e'\rel\oml e\rel\oml f\rel\oml f'$, so that $e'\rel\oml f'$, 
since $\oml$ is transitive. Hence we can unambiguously define a
relation $\le$ on $E/\gle$ by
\begin{equation*}  
  \gle(e)\le\gle(f)\;\;\text{if and only if}\;\;e\rel\oml f
\end{equation*}
It is not difficult to see that this relation is a partial order on
$E/\gle$. Also, it can be easily seen that for $e$ and $f$ in $E$, we
have $e\rel\oml f$ if and only if $Se\subseteq Sf$ and so
$\gle(e)\le\gle(f)$ if and only if $Se\subseteq Sf$.  Thus the
partially ordered set $E/\gle$ is isomorphic with the partially
ordered set of principal left ideals of $S$. Similarly, we can define
a partial order $\le$ on the set $E/\gre$ of $\gre$-classes in $E$ by
\begin{equation*}  
  \gre(e)\le\gre(f)\;\;\text{if and only if}\;\;e\rel\omr f
\end{equation*}
and this partially ordered set is isomorphic with the partially
ordered set of principal right ideals of $S$.

Thus if $E$ is the biordered set of idempotents of a strongly regular
Baer semigroup, then the quotients $E/\gle$ and $E/\gre$ are
complemented modular lattices and they are dually isomorphic. 
In the following, we denote these lattices by $\llat E$ and $\rlat E$
respectively. So, from Theorem~\ref{srbsg}, we get the following
\begin{cor}\label{llatcm}
  Let $E$ be a regular biordered set satisfying \ref{e1} and \ref{e2} of
  Theorem~\ref{srbsg}. Then $\llat E$ and $\rlat E$ are dually
  isomorphic complemented modular lattices.\qed
\end{cor}
We also note the following result on complements in 
$E/\gle$ 
\begin{cor}\label{llatcompl}
  Let $E$ be a regular biordered set satisfying \ref{e1} and
  \ref{e2} of Theorem~\ref{srbsg}. Then for each $e$ in $E$, the sets
  $\gle(e)$ and $\gle(\annid e)$ are complements of each other
  in the lattice $\llat E$.
\end{cor}
\begin{proof}
  By Theorem~\ref{srbsg}, there exists a strongly regular Baer semigroup
  $S$ with $E$ as its biordered set of idempotents. Let $e$
  be an element of $E$. We will show that $Se$ and
  $S\annid e$ are complements of each other in the lattice of
  principal left ideals of $S$. 
  
  Let $x\in Se\medcap S\annid e$. Then $xe=x$, since $x\in Se$. Also,
  $x\in S\annid e$ and by Lemma~\ref{ann}, we have $S\annid e=\lann
  e$, so that $xe=0$. Thus $x=xe=0$ and it follows that $Se\medcap
  S\annid e=\{0\}$.
  To show that $Se\medvee S\annid e=S$, suppose that 
  $Se\medvee S\annid e=Sf$, so that $Se\subseteq Sf$
  and $S\annid e\subseteq Sf$, which means $e\rel\oml f$ and 
  $\annid e\rel\oml f$. Since $e\rel\oml f$, we have 
  $\annid f\rel\omr\annid e$, so that $\annid e\annid f=\annid f$ and
  since $\annid e\rel\oml f$ we have $\annid f\rel\omr \dannid e=e$,
  so that $e\annid f=\annid f$. Hence
  \begin{equation*}
    \annid f=\annid e\annid f=\annid e(e\annid f)
    =(\annid ee)\annid f=0
  \end{equation*}
  and so $f=\dannid f=\annid 0=1$. Since $1$ is the identity of
  $S$, we have $S1=S$. Thus 
  $Se\medvee S\annid e=Sf=S1=S$.
\end{proof}

Since the multiplicative semigroup of a regular ring with unity is a
strongly regular Baer semigroup, these results  holds in particular for biordered
sets of regular rings. Now in \cite{vn}, it is shown that if $L$ is a
complemented modular lattice satisfying certain conditions, then it can
be realized as the lattice of principal left ideals of a matrix ring
over a regular ring. To translate these conditions into biorder terms,
we take a look at the idempotents in such a ring.

\section{Idempotents in matrix rings}

In \cite{vn}, it is shown that if $R$ is a regular ring, then for
every natural number $n$, the ring $R_n$ of $n\times n$ matrices over
$R$ is also a regular ring. In this section, we look at some
peculiarities of the biordered set of $R_n$. We first note that this
ring contains a special class of idempotents. For each
$i=1,2,\dotsc,n$ we define $\matre i$ to be the $n\times n$ matrix
with a single 1 at the $i^\text{th}$ row and $i^\text{th}$ column and
0's elsewhere. It easily follows from the usual rules of matrix
multiplication that $\matre i$ is an idempotent for each $i$. Also,
$\matre i\matre j=\matr0$ for $i\ne j$ and
$\matre1+\matre2+\dotsb+\matre n=\matr I$, where $\matr0$ is the
$n\times n$ zero matrix and $\matr I$ is the $n\times n$ identity
matrix. We now see how the condition on the sum of these idempotents
can be translated into biorder terms.
\begin{prop}
  Let $e_1, e_2,\dotsc,e_n$ be idempotents in a regular ring $R$ with
  unity such that $e_ie_j=0$ for $i\ne j$. Then the following are
  equivalent
  \begin{mathenum}
  \item $e_1+e_2+\dotsb+e_n=1$
  \item If $e$ is an idempotent in $R$ such that $e_i\rel\om e$ for
    each $i=1,2,\dotsc,n$, then $e=1$
  \end{mathenum}
\end{prop}
\begin{proof}
  First suppose that (i) holds and suppose that $e$ is an idempotent
  in $R$ with $e_i\rel\om e$ for $i=1,2,\dotsc,n$. Then $ee_i=e_i$ for
  each $i=1,2,\dotsc,n$, so that
  \begin{equation*}
    e=e1=e(e_1+e_2+\dotsb+e_n)=e_1+e_2+\dotsb+e_n=1
  \end{equation*}
  which gives (ii). 

  Conversely, suppose that (ii) holds and let
  $e=e_1+e_2+\dotsb+e_n$. Then
  \begin{equation*}
    e^2=(e_1+e_2+\dotsb+e_n)(e_1+e_2+\dotsb+e_n)=e_1+e_2+\dotsb+e_n=e
  \end{equation*}
  since each $e_i$ is an idempotent and $e_ie_j=0$ for $i\ne j$. Thus
  $e$ is an idempotent. Moreover for each $e_i$
  \begin{equation*}
    e_ie=e_i(e_1+e_2+\dotsb+e_n)=e_i
  \end{equation*}
  and similarly, $ee_i=e_i$. Thus $e_i\rel\om e$ for each $i$ and so
  $e=1$, by (ii).  
\end{proof} 

This discussion, together with Lemma~\ref{prod0}, gives the following 
\begin{prop}
  Let $R$ be a regular ring with unity and let $R_n$ be the ring of
  $n\times n$ matrices over  $R$. Then there exists idempotents
  $\matre1, \matre2,\dotsc\matre n$ in $R_n$ such that
  \begin{mathenum}
  \item $\mset{\matre i}{\matre j}=\{\matr 0\}$, for $i\ne j$
  \item if $\matre{}$ is an idempotent in $R_n$ such that 
    $\matre i\rel\om\matre{}$ for each $i=1,2,\dotsc,n$, then
    $\matre{}=\matr I$
    \qed
  \end{mathenum} 
\end{prop}

Another property of these idempotents is that any pair of them
generate principal left ideals which have a common complement (see the
proof of Theorem 3.3, Part II,\cite{vn}). To describe this property in
biorder terms, we introduce some terminology from \cite{kss} in a
slightly modified form. 

Let $e$ and $f$ be idempotents in a biordered set $E$. As in
\cite{kss}, by an $E$-sequence from $e$ to $f$, we mean a finite
sequence $e_0=e,e_1,e_2,\dotsc,e_{n-1},e_n=f$ of elements of $E$ such
that $e_{i-1}(\gle\medcup\gre)e_i$ for $i=1,2,\dotsc,n$ and in this
case, $n$ is called the length of the $E$-sequence. If there exists an
$E$-sequence from $e$ to $f$, we define $d(e,f)$ to be the length of
the shortest $E$-sequence from $e$ to $f$; also we define $d(e,e)=1$.
If there is no $E$-sequence form $e$ to $f$,we define $d(e,f)=0$.  For
our purposes, we will have to distinguish between $E$-sequences
starting with $\gle$ and those starting with $\gre$. For idempotents
$e$ and $f$, we define $d_l(e,f)$ to be the length of the shortest
$E$-sequence from $e$ to $f$, which start with the $\gle$ relation and
$d_r(e,f)$ to be the length of the shortest $E$-sequence from $e$ to
$f$ which start with the $\gre$ relation.

The condition for two principal left ideals of a ring to have a common
complement can be described in terms of the $d_l$ function as follows.
Following \cite{vn}, two elements of a lattice which have a common
complement are said to be \emph{in perspective}.
\begin{prop}\label{idpersp}
  Let $E$ be the set of idempotents of a regular ring $R$ and
  $e$ and $f$ be elements of $E$. Then $\gle(e)$ and $\gle(f)$
  are in perspective in $\llat E$ if and only if 
  $1\le d_l(e,f)\le3$.   
\end{prop}
\begin{proof}
  First suppose that $\gle(e)$ and $\gle(f)$ are in perspective and
  let $\gle(g)$ be a common complement of $\gle(e)$ and $\gle(f)$ in
  $\llat E$. Since $\gle(e)$ and $\gle(g)$ are complements of
  each other in $\llat E$, there exists $h$ in $E$ with $Rh=Re$ and 
  $R(1-h)=Rg$ (see \cite{vn}, Part II, Theorem~2.1) so that
  \begin{equation*}
    \gle(h)=\gle(e)
    \quad\text{and}\quad
    \gle(1-h)=\gle(g)
  \end{equation*}
  Again, since $\gle(f)$ and
  $\gle(g)$ are complements of each other, there exists $k$ in
  $E$ 
  \begin{equation*}
    \gle(k)=\gle(f)
    \quad\text{and}\quad 
    \gle(1-k)=\gle(g)
  \end{equation*}
  Now since $\gle(e)=\gle(h)$, we have $e\rel\gle h$ and since
  $\gle(k)=\gle(f)$, we have $k\rel\gle f$. Also, we have
  $\gle(1-h)=\gle(g)=\gle(1-k)$ so that $1-h\rel\gle 1-k$ and hence
  $h\rel\gre k$, by definition of the $\gle$-relation and
  Proposition~\ref{annid}. It follows from the definition of $d_l$
  that $1\le d_l(e,f)\le3$.
  
  Conversely, suppose $e$ and $f$ are elements of $E$ with $1\le
  d_l(e,f)\le 3$. Then there exist $g$ and $h$ in $E$ with $e\rel\gle
  g\rel\gre h\rel\gle f$ (where some of the elements may be equal).
  Since $e\rel\gle g$, we have $\gle(e)=\gle(g)$ and so $\gle(1-g)$ is
  a complement of $\gle(g)=\gle(e)$. Also, from $g\rel\gre h$, we have
  $1-g\rel\gle 1-h$ so that $\gle(1-g)=\gle(1-h)$ and so $\gle(1-g)$
  is a complement of $\gle(h)$. Moreover, from $h\rel\gle f$, we have
  $\gle(h)=\gle(f)$.  Hence $\gle(1-g)$ is a complement of
  $\gle(h)=\gle(f)$. Thus $\gle(1-g)$ is a complement of both
  $\gle(e)$ and $\gle(f)$.
\end{proof}

We next show that any regular biordered set satisfying some of the
conditions discussed so far can be realized as the set of idempotents
of a ring of matrices over a regular ring.

\section{Biordered sets of matrix rings}

In this section, we prove our main result:
\begin{thm}\label{mr}  Let $E$ be a regular biordered set
  satisfying the following 
  properties.
  \renewcommand{\theenumi}{{\normalfont(\textsf{E\arabic{enumi}})}}
  \renewcommand{\labelenumi}{\theenumi}
  \renewcommand{\theenumii}{{\normalfont(\roman{enumii})}}
  \renewcommand{\labelenumii}{\theenumii}
  \begin{enumerate}
  \item There exists an element $0$ in $E$ such
    that $0\rel\om e$ for each $e$ in $E$
  \item There exists a map 
    $e\mapsto\annid e$ satisfying the following conditions:
    \begin{enumerate}
    \item $\dannid e=e$ for each $e$ in $E$
    \item $f\rel\oml e$ if and only if $\annid e\rel\omr\annid f$ for
      $e$, $f$ in $E$
    \item $f\rel\oml\annid e$ if and only if $\mset fe=\{0\}$ for $e$,
      $f$ in $E$
    \end{enumerate}
  \item If $f\rel\om\annid e$, then 
    $\swset{\annid e}{\annid f}\medcap\swset{\annid f}{\annid e}\ne\emptyset$
    \label{e3}
  \item There exists idempotents $e_1,e_2,\dotsc,e_n$ in $E$,
    where $n\ge4$, satisfying the following conditions:\label{e4}
    \begin{enumerate}
    \item \mset{e_i}{e_j}=\{0\}, for $i\ne j$\label{e41}
    \item if $e$ is in $E$ with $e_i\rel\om e$ for $i=1,2,\dotsc,n$,
      then $e=1$\label{e42}
    \item $d_l(e_i,e_j)=3$, for $i\ne j$\label{e43}
    \end{enumerate}
  \end{enumerate}
  Then there exists a regular ring $R$ with the biordered set of
  idempotents of the ring $R_n$ of $n\times n$ matrices over $R$
  isomorphic with $E$. 
\end{thm} 
To prove this result, we first look at some consequences of these
conditions. We start by noting that in the case of a biordered set
satisfying \ref{e1}, \ref{e2} and \ref{e3}, there is exactly one
element in 
$\swset{\annid e}{\annid f}\medcap\swset{\annid f}{\annid e}$.
\begin{prop}\label{oplus}
  Let $E$ be a regular biordered set satisfying \ref{e1}, \ref{e2} and
  \ref{e3}. If $f\rel\om \annid e$ then there is a unique element in
  $E$ belonging to  
  $\swset{\annid e}{\annid f}\medcap\swset{\annid f}{\annid e}$.
\end{prop}
\begin{proof}
  By Theorem~\ref{srbsg}, there exists a regular semigroup
  $S$ with its biordered set of idempotents equal to $E$. Since
  $S$ is regular, we have 
  \begin{equation*}
    \swset ef=\{h\in E\colon ehf=ef\;\;\text{and}\;\;fhe=h\}  
  \end{equation*}
  (see \cite{kss}, Theorem 1.1). Let 
  $h\in\swset{\annid e}{\annid f}\medcap\swset{\annid f}{\annid e}$.  
  Then $h\in \swset{\annid e}{\annid f}$, 
  so that
  \begin{equation*}
    \annid e h\annid f=\annid e\annid f
  \end{equation*}
  Also, $h\in\swset{\annid f}{\annid e}$,
  so that 
  \begin{equation*}
    \annid e h=h\annid f=h
  \end{equation*}
  Hence
  \begin{equation*}
    h=h^2=(\annid eh)(h\annid f)=\annid eh\annid f=\annid e\annid f
  \end{equation*}
  Thus the only element in 
  $\swset{\annid e}{\annid f}\medcap\swset{\annid f}{\annid e}$ is 
  $\annid e\annid f$.
\end{proof}

In the following, for $e$ and $f$ in a biordered set satisfying
\ref{e1}, \ref{e2} and \ref{e3}, if $h$ is the unique element of
$\swset{\annid e}{\annid f}\medcap\swset{\annid f}{\annid e}$, then we
denote $\annid h$ by $e\oplus f$. The next result gives an alternate
characterization of $e\oplus f$.
\begin{prop}\label{sumidalt}
  Let $e$ and $f$ be elements of a regular biordered set satisfying \ref{e1},
  \ref{e2} and \ref{e3} with $f\rel\om\annid e$ and let $h=e\oplus f$.
  Then $h$ satisfies the following conditions.
  \begin{mathenum}
  \item $e\rel\om h$ and $f\rel\om h$
  \item if $g$ is in $E$ with $e\rel\oml g$ and $f\rel\oml g$,
    then $h\rel\oml g$
  \item if $g$ is in $E$ with $e\rel\omr g$ and $f\rel\omr g$,
    then $h\rel\omr g$
  \end{mathenum}
  Moreover, these properties characterize $h$.
\end{prop}
\begin{proof}
  To prove (i), note that $\annid{h}\in
  \swset{\annid{e}}{\annid{f}}\medcap\swset{\annid{f}}{\annid{e}}$, by
  definition.  Since $\annid h\in\swset{\annid e}{\annid f}$, we have
  $\annid h\rel\oml\annid e$ and $\annid h\rel\omr\annid f$, so that
  $e\rel\omr h$ and $f\rel\oml h$. Similarly, since $\annid
  h\in\swset{\annid f}{\annid e}$, we have $e\rel\oml h$ and
  $f\rel\omr h$. Thus $e\rel\om h$ and $f\rel\om h$.
 
  To prove (ii), let $g\in E$ with $e\rel\oml g$ and 
  $f\rel\oml g$. Then $\annid g\rel\omr \annid e$ and 
  $\annid g\rel\omr\annid f$. Hence
  \begin{equation*}
    \annid e\annid g=\annid g
    \quad\text{and}\quad
    \annid f\annid g=\annid g
  \end{equation*}
  Let $S$ be a regular semigroup with its biordered
  set of idempotents equal to $E$. Then as seen in the previous
  result, we have $\annid h=\annid e\annid f$. Hence
  \begin{equation*}
    \annid h\annid g=(\annid e\annid f)\annid g
    =\annid e(\annid f\annid g)
    =\annid e\annid g
    =\annid g
  \end{equation*}
  so that $\annid g\rel\omr\annid h$ and so $h\rel\oml g$. This proves
  (ii). A dual argument establishes (iii).
  
  To prove uniqueness of $h$, suppose $h$ and $h'$ are elements of
  $E$ satisfying these conditions. Then $e\rel\oml h'$ and
  $f\rel\oml h'$, so that $h\rel\oml h'$. Similarly $h\rel\omr h'$ so
  that $h\rel\om h'$. Interchanging the roles of $h$ and $h'$, we also
  have $h'\rel\om h$. Thus $h'=h$.
\end{proof}

Now let $e$ and $f$ be elements of $E$ with $f\rel\om\annid e$,
so that we have $e\oplus f$ in $E$. Let $h=e\oplus f$. Then by
the above result, $e\rel\oml h$ and $f\rel\oml h$, so that in the
lattice $\llat E =E/\gle$, we have $\gle(e)\le\gle(h)$ and 
$\gle(f)\le\gle(h)$. Also, if $g\in E$ with $\gle(e)\le\gle(g)$
and $\gle(f)\le\gle(g)$, then $e\rel\oml g$ and $f\rel\oml g$, so
that $h\rel\oml g$ so that $\gle(h)\le\gle(g)$. It follows that 
$\gle(e)\medvee\gle(f)=\gle(h)$. 

Also, in this case, $\gle(e)\medcap\gle(f)=\{0\}$. For suppose
$g\in\gle(e)\medcap\gle(f)$ so that $ge=g$ and $gf=g$, and so in any
regular idempotent generated semigroup $S$ with $E$ as the
biordered set of idempotents,
\begin{equation*}
  g=ge=(gf)e=g(fe)
\end{equation*}
Also since $f\rel\oml\annid e$, we have $\mset{f}{e}=\{0\}$ and so
$fe=0$, by Lemma~\ref{prod0}. Hence $g=g(fe)=0$.  Thus we have the result
below:
\begin{prop}\label{sumidlat}
  Let $E$ be a regular biordered set satisfying \ref{e1}, \ref{e2} and
  \ref{e3}. Then for $e$ and $f$ in $E$ with $f\rel\om\annid e$
  we have $\gle(e)\medvee\gle(f)=\gle(e\oplus f)$ and
  $\gle(e)\medcap\gle(f)=\{0\}$ in the lattice 
  $\llat E =E/\gle$.\qed
\end{prop}

The above result can be extended. Let $E$ be as before and let
$S$ be an idempotent generated regular semigroup with $E$ as
the biordered set of idempotents. Suppose $e_1$, $e_2$, $e_3$ be
elements of $E$ with $\mset{e_i}{e_j}=\{0\}$ for $i\ne j$.
Since $\mset{e_1}{e_2}=\{0\}$, we have $e_1\rel\oml\annid{e_2}$ and
since $\mset{e_2}{e_1}=\{0\}$, we have $e_2\rel\oml\annid{e_1}$,
which implies $e_1\rel\omr\annid{e_2}$. Thus $e_1\rel\om\annid{e_2}$
and so we have $f_1=e_1\oplus e_2$ in $E$. In the same fashion,
since $\mset{e_1}{e_3}=\{0\}$ and $\mset{e_2}{e_3}=\{0\}$, we have
$e_1\rel\oml\annid{e_3}$ and $e_2\rel\oml\annid{e_3}$, so that
$f_1=e_1\oplus e_2\rel\oml\annid{e_3}$, by Proposition~\ref{sumidalt}. Dually, we
have $f_1\rel\omr\annid{e_3}$. Thus $f_1\rel\om\annid{e_3}$ and so we
have $f_1\oplus e_3$ in $E$. As in the proof of Proposition~\ref{sumidalt},
we can show that this element of $E$ is the least upper bound of
$e_1$, $e_2$ and $e_3$ with respect to $\oml$ and $\omr$ and so is
uniquely determined by these elements. Hence we can unambiguously write
$f_1\oplus e_3$ as $e_1\oplus e_2\oplus e_3$. Also, as in
Proposition~\ref{sumidlat}, we have
\begin{equation*}
  \gle(e_1)\medvee\gle(e_2)\medvee\gle(e_3)=
  \gle(e_1\oplus e_2\oplus e_3) 
\end{equation*}
Again, for distinct $i$, $j$, $k$, we have $e_i\rel\oml\annid{e_k}$ and
$e_j\rel\oml e_k$ so that $(e_i\oplus e_j)\rel\oml\annid{e_k}$, so
that $\mset{e_i\oplus e_j}{e_k}=\{0\}$  and so $(e_i\oplus e_j)e_k=0$
in $S$. Hence
\begin{equation*}
  \left(\gle(e_i)\medvee\gle(e_j)\right)\medcap\gle(e_k)
  =\gle(e_i\oplus e_j)\medcap\gle(e_k)=\{0\}
\end{equation*}
By induction, we have the following result. Note that elements
$a_1,a_2,\dotsc,a_n$ of a lattice are said to be independent
if for each $i=1,2,\dotsc,n$, we have 
$a_i\medwedge\bigl(\medvee_{\substack{j=1\\j\ne i}}^na_j\bigr)=0$
\begin{prop}\label{llatind}
  Let $E$ be a regular biordered set satisfying \ref{e1}, \ref{e2}
  and \ref{e3} and let $e_1,e_2,\dotsc,e_n$ be elements of $E$
  with $\mset{e_i}{e_j}=\{0\}$ for $i\ne j$. Then
  $\gle(e_1),\gle(e_2),\dotsc,\gle(e_n)$ are independent elements in
  the lattice $\llat E=E/\gle$ with
  $\gle(e_1)\medvee\gle(e_2)\dotsb\medvee\gle(e_n)=
  \gle(e_1\oplus e_2\dotsb\oplus e_n)$.\qed
\end{prop}

We can show that as in the proof of Proposition~\ref{idpersp} that the condition 
\ref{e43} implies that $\gle(e_i)$ and $\gle(e_j)$ are in
perspective. 
\begin{prop}\label{llatpersp}
  Let $E$ be a regular biordered set satisfying \ref{e1} and \ref{e2}
  and let $e$ and $f$ be elements in $E$ with $d_l(e,f)\le3$. Then
  $\gle(e)$ and $\gle(f)$ are in perspective in the lattice
  $\llat E=E/\gle$. 
\end{prop}
\begin{proof}
  Since $d_l(e,f)\le3$, there exists $g$ and $h$ in $E$, with
  $e\rel\gle g\rel\gre h\rel\gle f$. Let $k=\annid g$. Then $k$ is in
  $E$ with $\annid k=g$. So, $\gle(k)$ is a complement of
  $\gle(\annid k)=\gle(g)$ in the lattice $\llat E$, by
  Corollary~\ref{llatcompl}. Also, since $g\rel\gle e$, we have $\gle(g)=\gle(e)$.
  Thus $\gle(k)$ is a complement of $\gle(e)$ in $\llat E$.

  Again, since $g\rel\gre h$, we have $\annid h\rel\gle \annid g=k$
  so that $\gle(k)=\gle(\annid h)$ is a complement of $\gle(h)$ in
  $\llat E$. Also, since $h\rel\gle f$, we have
  $\gle(h)=\gle(f)$. Thus $\gle(k)$ is a complement of $\gle(f)$ also
  in $\llat E$. 
\end{proof} 

Now suppose $E$ is a regular biordered set satisfying \ref{e1},
\ref{e2}, \ref{e3} and \ref{e4}. Then by Theorem~\ref{srbsg},
there exists a strongly regular Baer semigroup $S$ with its
biordered set of idempotents isomorphic with $E$. Also, the
lattice of principal left ideals of $S$ is isomorphic with
$E/\gle=\llat E$, so that $\llat E$ is a complemented modular lattice,
as seen in Corollary~\ref{llatcm} 

In \cite{pa}, it is shown how a biordered set
$E(L)$ can be constructed from a complemented modular
lattice $L$ and it is shown that if $S$ is a strongly
regular Baer semigroup with its lattice of principal left ideals
isomorphic with $L$, then its biordered set $E(S)$
is isomorphic with $E(L)$. (see Corollary 4 of \cite{pa})

Hence for our biordered set $E$, we have the strongly regular Baer
semigroup $S$ with its biordered set isomorphic with $E$
and lattice of principal left ideals isomorphic with $\llat E$, so
that by the result cited above, $E(\llat E)$ is isomorphic
with $E$.

Also, since $E$ satisfies \ref{e41}, the members
$\gle(e_1),\gle(e_2),\dotsc,\gle(e_n)$ form an independent set in
$\llat E$, by Proposition~\ref{llatind}. Moreover, by the same result, if
$h=e_1\oplus e_2\oplus\dotsb\oplus e_n$, then
$\gle(e_1)\medvee\gle(e_2)\dotsb\gle(e_n)=\gle(h)$. Now $e_i\rel\om h$ for
each $i$, by Proposition~\ref{sumidalt} and so $h=1$, by \ref{e42}. Hence
$\gle(e_1)\medvee\gle(e_2)\dotsb \medvee \gle(e_n)=\gle(1)$. Also, \ref{e43}
implies that these members of $\llat E$ are in perspective, by
Proposition~\ref{llatpersp}. Thus this set is a \textit{homogeneous basis} of
$\llat E$, in the sense of \cite{vn}. Thus $\llat E$ is a
complemented modular lattice with a homogeneous basis of rank $n$ and
so if $n\ge 4$, then there exists a regular ring $R$ with the
lattice of principal left ideals of the matrix ring $R_n$
isomorphic with $\llat E$.(see Theorem 14.1 of \cite{vn})

Again in \cite{pa}, it is shown  (see Theorem 5 of \cite{pa}) that
if the lattice of principal left ideals of a regular ring is isomorphic
with $L$, then the biordered set of the ring is isomorphic
with $E(L)$. Hence in our case, the biordered set of idempotents of the
regular ring $R_n$ is isomorphic with $E(\llat E)$ which is
isomorphic with $E$. This proves our theorem.

\end{document}